\documentclass[a4,12pt]{amsart}
\usepackage{amssymb,amstext,amsmath,amscd,amsthm,amsfonts,enumerate,latexsym}
\usepackage{color}
\usepackage{graphicx}
\usepackage[all]{xy}
\theoremstyle{plain}
\newtheorem{thm}{Theorem}[section]

\newtheorem*{thm*}{Theorem}
\newtheorem*{cor*}{Corollary}

\newtheorem{prop}[thm]{Proposition}
\newtheorem{proposition}[thm]{Proposition}

\newtheorem{claim}{Claim}
\newtheorem*{claim*}{Claim}

\theoremstyle{definition}
\newtheorem{defn}[thm]{Definition}

\newtheorem{ex}[thm]{Example}

\newtheorem{prob}[thm]{Problem}

\theoremstyle{remark}

\numberwithin{equation}{thm}

\def\GCD{\operatorname{GCD}}
\def\Ext{\operatorname{Ext}}

\def\Ker{\operatorname{Ker}}

\def\Hom{\operatorname{Hom}}

\def\mod{\mathrm{mod}}

\def\a{\mathrm a}

\newcommand{\rma}{\mathrm{a}}
\newcommand{\rmf}{\mathrm{f}}

\newcommand{\rmr}{\mathrm{r}}\newcommand{\rmt}{\mathrm{t}}

\newcommand{\rmI}{\mathrm{I}}\newcommand{\rmK}{\mathrm{K}}

\newcommand{\rmQ}{\mathrm{Q}}

\newcommand{\fkm}{\mathfrak{m}}

\newcommand{\fkM}{\mathfrak{M}}

\def\height{\mathrm{ht}}

\def\PF{\operatorname{PF}}

\title[Pseudo-Frobenius numbers versus defining ideals]{Pseudo-Frobenius numbers versus defining ideals \\
in numerical semigroup rings}
\author{Shiro Goto}
\address{Department of Mathematics, School of Science and Technology, Meiji University, 1-1-1 Higashi-mita, Tama-ku, Kawasaki 214-8571, Japan}
\email{shirogoto@gmail.com}
\author{Do Van Kien}
\address{Department of Mathematics, Hanoi Pedagogical University N$^{0}2$, Vinh Phuc, Vietnam}
\email{dovankien@hpu2.edu.vn}
\author{Naoyuki Matsuoka}
\address{Department of Mathematics, School of Science and Technology, Meiji University, 1-1-1 Higashi-mita, Tama-ku, Kawasaki 214-8571, Japan}
\email{naomatsu@meiji.ac.jp}
\author{Hoang Le Truong}
\address{Institute of Mathematics, Vietnam Academy of Science and Technology, 18 Hoang Quoc Viet Road, 10307 Hanoi, Vietnam}
\email{hltruong@math.ac.vn}
\thanks{The first author was partially supported by JSPS Grant-in-Aid for Scientific Research (C) 16K05112. The second author was partially supported by the International Research Supporting Program of Meiji University. The third author was partially supported by JSPS Grant-in-Aid for Scientific Research 26400054. The fourth author was partially supported by JSPS bilateral programs (Joint Research) and the Vietnam National Foundation for Science and Technology Development (NAFOSTED) under grant number 101.04-2017.14.}

\begin{document}
\maketitle

\setlength{\baselineskip}{17pt}

\begin{abstract}
The structure of the defining ideal of the semigroup ring $k[H]$ of a numerical semigroup $H$ over a field $k$ is described, when the pseudo-Frobenius numbers of $H$ are multiples of a fixed integer.
\end{abstract}



\section{Introduction}\label{section1}
The study of numerical semigroup rings is one of the most fruitful sources for researches, not only about numerical semigroups but also about commutative algebra. As for the latter case, in \cite{H} J. Herzog provided new possible areas of study for one-dimensional Cohen-Macaulay rings. He showed that the defining ideals of semigroup rings for $3$-generated numerical semigroups are generated by the maximal minors of $2 \times 3$  matrices, provided the defining ideals are not complete intersections. This was a glorious departure of combinatorial commutative algebra. Besides, his result has been used, e.g., to compute the symbolic powers of the defining ideals for space monomial curves, and proved very useful in order to produce counter examples to Cowsik's question (\cite{GNW}). Since the article [H], many authors have been interested, e.g., in the question of estimating the number of generators of the defining ideal of a given numerical semigroup ring. Among many striking works, in \cite{B}  H. Bresinski succeeded in describing the structure of the defining ideals of $4$-generated Gorenstein numerical semigroup rings. It should be noted here that even for the $4$-generated case, in general the number of generators of the defining ideals can be as large as one needs, and the behavior of defining ideals is rather wild, unless any specific conditions on the semigroups are equipped.

The purpose of the present article is to study the structure of the defining ideals in the case where the pseudo-Frobenius numbers of the semigroup are multiples of a fixed integer. To explain our motivation and the meaning of our result, we need some notation. In what follows,
 let $a_1, a_2, \ldots , a_n \in \mathbb{Z}$ be positive integers such that $\GCD(a_1, a_2, \ldots, a_n) =1$. Let $$H=\left<a_1, a_2, \ldots , a_n\right> = \left\{\sum_{i=1}^n c_i a_i \  \middle| \  0 \le c_i \in \mathbb{Z} \text{ for all }1\le i \le n \right\}$$ be the numerical semigroup generated by the integers $a_i$'s. Let $k$ be a field. We set $R=k[H] = k[t^{a_1}, t^{a_2}, \ldots , t^{a_n}]$ and call it the numerical semigroup ring of $H$ over $k$, where $t$ is an indeterminate. Let $S=k[x_1,x_2, \ldots , x_n]$ be the weighted polynomial ring over $k$ with $\deg x_i = a_i$ for each $1 \le i \le n$. Let $\varphi: S \to R$ denote the homomorphism of graded $k$-algebras defined by $\varphi(x_i) = t^{a_i}$ for all $1 \le i \le n$. Let $I=\Ker \varphi$ be the defining ideal of $R$. With this notation, in this article we are interested in the following problem. 

\begin{prob}
What kind of structures does the ideal $I$ enjoy?
\end{prob}

As stated above, when the number $n$ is small, a few answers are already known. The main result for the case of $n=3$ is due to Herzog (\cite{H}). When $n=4$, the problem suddenly becomes complicated, and there are only partial answers. We refer to two articles: one is \cite{B} and the other one is \cite{K}, where J. Komeda gave a description of a minimal system of generators of $I$ in the case where the semigroup $H$ is pseudo-symmetric.

Our aim of the present article is to find the connection with the behavior of pseudo-Frobenius numbers of $H$ and the generation of the defining ideal of $R=k[H]$, where the embedding dimension $n$ of $R$ is arbitrary. To state our result, we need more notation. Let $\rmf(H) = \max(\mathbb{Z} \setminus H)$ denote the Frobenius number of $H$. We set $$\PF(H) = \{\alpha \in \mathbb{Z} \setminus H \mid \alpha + a_i \in H \text{ for all } 1 \le i \le n\}$$ and call the elements in $\PF(H)$ pseudo-Frobenius numbers of $H$. Hence $\rmf(H) \in \PF(H)$ and $$\rmK_R = \sum\limits_{\alpha \in \PF(H)} R t^{-\alpha}$$ (\cite{GW}), where  $\rmK_R$ denotes the graded canonical module of $R$. Therefore, the $a$-invariant $\mathrm{a}(R)$ of $R$ (resp. the Cohen-Macaulay type $\rmr(R)$ of $R$) is given by $\rma(R) = \rmf(H)$ (resp. $\rmr(R) = \sharp \PF(H)$). For a given matrix $A$ with entries in $S$, we denote by $\rmI_2(A)$ the ideal of $S$ generated by $2 \times 2$ minors of $A$. With this notation, the main result of the present article is stated as follows.

\begin{thm}\label{MainThm}
Let $H=\left<a_1, a_2, \ldots, a_n\right>$~$($$n \ge 3$$)$ be a numerical semigroup and assume that $H$ is minimally generated by the $n$ numbers $\{a_i\}_{1 \le i \le n}$. Then the following conditions are equivalent.
\begin{enumerate}
\item[{\rm (1)}] $I=\rmI_2 \begin{pmatrix} f_1 & f_2 & \cdots & f_n\\x_1 & x_2 & \cdots & x_n\end{pmatrix}$ for some homogeneous elements $f_1, f_2, \ldots, f_n \in S_+=(x_i \mid 1 \le i \le n)$.
\item[{\rm (2)}] After suitable permutations of $a_1, a_2, \ldots , a_n$  if necessary, we have
$$
I=\rmI_2 \begin{pmatrix} 
x_2^{\ell_2} &  x_3^{\ell_3}& \cdots & x_n^{\ell_n} & x_1^{\ell_1} \\
x_1 &  x_2 & \cdots & x_{n-1} & x_n\end{pmatrix}
$$
for some positive integers $\ell_1, \ell_2, \ldots ,\ell_n >0$.
\item[{\rm (3)}] There exists an element $\alpha \in \PF(H)$ such that $(n-1)\alpha \notin H$.
\end{enumerate}
When this is the case, the following assertions hold true.
\begin{enumerate}
\item[{\rm (a)}] For each $1 \le i \le n$, we have $\ell_i = \min\{\ell >0 \mid \ell a_i \in H_i\} -1$, where $$H_i =\left<a_1, \ldots, \overset{\vee}{a_i}, \ldots, a_n\right>.$$
\item[{\rm (b)}] $\alpha = \deg f_i - a_i$ for all $1 \le i \le n$.
\item[{\rm (c)}] $\PF(H) = \{\alpha, 2\alpha, \ldots , (n-1)\alpha\}$. 
\item[{\rm (d)}] The ring $R=k[H]$ is an almost Gorenstein graded ring, i.e., the numerical semigroup $H$ is almost symmetric.
\end{enumerate}
\end{thm}

With the notation of Theorem \ref{MainThm}, suppose that $n=3$ and $R=k[H]$ is not a Gorenstein ring. We then have $\rmr(R) = 2$, and the defining ideal of $R$ is generated by the $2 \times 2$ minors of a matrix of the form
$$
\begin{pmatrix}
x_1^\alpha & x_2^\beta & x_3^\gamma\\
x_2^{\beta'} & x_3^{\gamma'} & x_1^{\alpha'}
\end{pmatrix}
$$
where $\alpha, \beta, \gamma, \alpha', \beta', \gamma' > 0$ (\cite{H}). By \cite{GMP, NNW} we see that  $R$ is an almost Gorenstein graded ring if and only if either $\alpha = \beta=\gamma =1$ or $\alpha' = \beta' = \gamma'=1$, while it is fairly well-known that $R$ is an almost Gorenstein graded ring if and only if $\PF(H) = \{\frac{\rmf(H)}{2}, \rmf(H)\}$ (see Proposition \ref{charAG}). Theorem \ref{MainThm} provides the case of higher embedding dimension $n \ge 3$ with an extension of these equivalences.

Let us explain how this paper is organized. The proof of Theorem \ref{MainThm} shall be given in Section 5. In Sections 2, 3, and 4,  we summarize some preliminaries, which we use to prove Theorem \ref{MainThm}. Section 2 consists of a survey on the almost Gorenstein numerical semigroup rings.  In Section 3, we review row-factorization matrices (RF-matrices for short) introduced by A. Moscariello \cite{M}. In Section 4 a remark on the Eagon-Northcott complex associated to a certain $2 \times n$ matrix consisting of homogeneous polynomials, which plays a key role in our proof of Theorem \ref{MainThm}. In Section 6, we give a few examples in order to illustrate Theorem \ref{MainThm}.


\section{Almost Gorenstein numerical semigroup rings}

The notion of an almost Gorenstein ring was introduced in 1997 by V. Barucci and R. Fr\"{o}berg \cite{BF} for one-dimensional analytically unramified local rings, where they deeply studied numerical semigroup rings, starting a beautiful theory. The first author, third author, and T. T. Phuong \cite{GMP} relaxed the notion for arbitrary Cohen-Macaulay local rings of dimension one, and subsequently, the first author, R. Takahashi, and N. Taniguchi \cite{GTT} gave the definition of an almost Gorenstein local/graded ring of higher dimension. The present interests are focused on the one-dimensional case. Therefore, we recall \cite[Definition 3.3]{GTT} in the following form.

\begin{defn}\label{defAGL}
Let $(R, \fkm)$ be a Cohen-Macaulay local ring of dimension one, possessing the  canonical module $\rmK_R$. Then we say that $R$ is an almost Gorenstein local ring, if there exists an exact sequence
$$
0 \to R \to \rmK_R \to C \to 0
$$
of $R$-modules such that $\fkm C=(0)$.
\end{defn}

Let $\overline{R}$ denote the integral closure of $R$ in the total ring $\rmQ(R)$ of fractions of $R$ and suppose that there is a fractional ideal $K$ of $R$ such that $R \subseteq K \subseteq \overline{R}$ and $K \cong \rmK_R$ as an $R$-module. Then the condition required in Definition \ref{defAGL} is equivalent to saying that $\fkm K \subseteq R$ (see \cite[Definition 3.1]{GMP} and \cite[Proposition 3.4]{GTT}).

\begin{defn}[\cite{GTT}]\label{defAGG}
Let $R = \bigoplus_{n \ge 0}R_n$ be a Cohen-Macaulay graded ring of dimension one such that $R_0$ is a local ring. Suppose that $R$ possesses the graded canonical module $\rmK_R$. Let $\fkM$ denote the unique graded maximal ideal of $R$ and $a=\rma(R)$ the a-invariant of $R$. We say that $R$ is an almost Gorenstein graded ring, if there exists an exact sequence
$$
0 \to R \to \rmK_R(-a) \to C \to 0
$$
of graded $R$-modules such that $\fkM C=(0)$. Here $\rmK_R(-a)$ stands for the graded $R$-module whose underlying $R$-module is the same as that of $\rmK_R$ and whose grading is given by $[\rmK_R(-a)]_n = [\rmK_R]_{n-a}$ for all $n \in \mathbb{Z}$.
\end{defn}

Consequently, every Gorenstein local/graded ring is an almost Gorenstein ring. The condition in Definition \ref{defAGL} (resp. Definition \ref{defAGG}) asserts that once $R$ is an almost Gorenstein local (resp. graded) ring, either $R$ is a Gorenstein ring or, even though $R$ is not a Gorenstein ring, the local (resp. graded) ring $R$ is embedded into the module $\rmK_R$ (resp. the graded module $\rmK_R(-a)$) and the difference $C$ is a vector space over $R/\fkm$ (resp. $R/\fkM$).

If $R$ is an almost Gorenstein graded ring, then the localization $R_\fkM$ of $R$ by $\fkM$ is an almost Gorenstein local ring, which readily follows from the definition. The converse does not hold true in general (\cite[Example 8.8]{GTT}, \cite[Theorems 2.7, 2.8]{GMTY}). However, it does for numerical semigroup rings, as we confirm in the following.

\begin{prop}[{cf. \cite{GMP, GTT, N, NNW}}]\label{charAG}
Let $R=k[H]$ be the semigroup ring of a numerical semigroup $H = \left<a_1, a_2, \ldots ,a_n\right>$ over a field $k$ and let $\fkM = (t^{a_1}, t^{a_2}, \ldots, t^{a_n})$ be the graded maximal ideal of $R$. Let us write $\PF(H) = \{\alpha_1, \alpha_2, \ldots , \alpha_r\}$, so that $\alpha_1 < \alpha_2 < \cdots < \alpha_r$, whence $\alpha_r = \rmf(H)$. Then the following conditions are equivalent.
\begin{enumerate}[{\rm (1)}]
\item $R$ is an almost Gorenstein graded ring.
\item $R_\fkM$ is an almost Gorenstein local ring.
\item $\alpha_i + \alpha_{r-i} = f(H)$ for all $1 \le i \le r-1$.
\end{enumerate}
\end{prop}

\begin{proof}
We set $K=t^{f(H)} \rmK_R$. Then, $R \subseteq K \subseteq k[t] = \overline{R}$ and $K \cong \rmK_R(-f(H))$ as a graded $R$-module. Because $K/R$ is a graded $R$-module, $\fkM{\cdot}K \subseteq R$ if and only if $\fkM{\cdot}K_\fkM \subseteq R_\fkM$, and the former condition is equivalent to saying that $a_j +(f(H) - \alpha_i) \in H$ for all $1 \le i \le r-1$ and $1 \le j \le n$, i.e., $f(H) - \alpha_i \in \PF(H)$ for all $1 \le i \le r-1$.
\end{proof}

\section{RF-matrices associated to pseudo-Frobenius numbers}

The notion of an RF-matrix associated to a pseudo-Frobenius number $\alpha \in \PF(H)$ was introduced by A. Moscariello \cite{M} in 2016, and  by means of RF-matrices, Moscariello proved that the Cohen-Macaulay type $\rmt(H)$ of any $4$-generated numerical semigroup $H$ is at most $3$, if the semigroup ring $k[H]$ over a filed $k$ is an almost Gorenstein graded ring. Because RF-matrices play a very important role in the present paper, let us recall here the definition of RF-matrices also.

\begin{defn}[\cite{M}]\label{defRF}
Let $H = \left<a_1, a_2, \ldots , a_n\right>$ be a numerical semigroup and $\alpha \in \PF(H)$ a pseudo-Frobenius number of $H$. Then an $n \times n$ matrix $M = (m_{ij})$ of integers is said to be an RF-matrix associated to $\alpha$, if the following conditions are satisfied.
\begin{enumerate}[(1)]
\item $m_{ii} = -1$ for all $1 \le i \le n$, 
\item $m_{ij} \ge 0$ if $i \ne j$, and
\item $M{\cdot}\begin{pmatrix}a_1\\a_2\\\vdots\\a_n\end{pmatrix} =\begin{pmatrix}\alpha\\\alpha\\\vdots\\\alpha\end{pmatrix}$.
\end{enumerate}
\end{defn}

For each $\alpha \in \PF(H)$, one can associate at least one RF-matrix $M$  to $\alpha$. In fact, since $\alpha \in \PF(H)$, we have $\alpha + a_i \in H$ for each $1 \le i \le n$. Therefore
\begin{eqnarray*}
\alpha + a_1 &=& \ell_{11} a_1 + \ell_{12} a_2 + \cdots + \ell_{1n}a_n\\
\alpha + a_2 &=& \ell_{21} a_1 + \ell_{22} a_2 + \cdots + \ell_{2n}a_n\\
&\vdots&\\
\alpha + a_n &=& \ell_{n1} a_1 + \ell_{n2} a_2 + \cdots + \ell_{nn}a_n
\end{eqnarray*}
for some non-negative integers $\{\ell_{ij}\}_{1\le i,j \le n}$. If $\ell_{ii} > 0$ for some $1 \le i \le n$, then $$\alpha = \sum_{j \ne i }\ell_{ij} a_j + (\ell_{ii}-1)a_{i} \in H,$$ which contradicts the definition of pseudo-Frobenius numbers. Hence $\ell_{ii} = 0$ for all $1 \le i \le n$. We now set $m_{ii} = -1$ for $1 \le i \le n$ and $m_{ij} = \ell_{ij}$ for $i \ne j$. Then $M=(m_{ij})$ is a required RF-matrix associated to $\alpha$. For a given pseudo-Frobenius number, RF-matrices are not necessarily uniquely determined in general (\cite[Example 3]{M}).

We need the following to prove Theorem \ref{MainThm}

\begin{prop}\label{relRF}
Let $H = \left<a_1, a_2, \ldots , a_n\right>$ be a numerical semigroup. Let $\alpha \in \PF(H)$  and  let $M = (m_{ij})$ be an RF-matrix associated to $\alpha$. Let $k[H]$ be the semigroup ring over a field $k$. Then 
$$
\rmI_2\begin{pmatrix}
f_1 & f_2 & \cdots & f_n\\
x_1 & x_2 & \cdots & x_n
\end{pmatrix} \subseteq \Ker \varphi,
$$
where $\Ker \varphi$ stands for the defining ideal of $k[H]$ in the polynomial ring $k[x_1, x_2, \ldots, x_n]$ and $f_i = \prod_{j \ne i} x_j^{m_{ij}}$ for each $1 \le i \le n$.
\end{prop}

\begin{proof}
It is enough to show $\deg x_{i_1} f_{i_2} = \deg x_{i_2} f_{i_1}$ for $1 \le i_1 < i_2 \le n$. Note that $$\alpha = \sum_{j \ne i_1} m_{i_1j} a_j -a_{i_2} = \sum_{j \ne i_2} m_{i_2j} a_j - a_{i_1}$$ by the definition of the RF-matrix $M$.
Hence $$\deg x_{i_1} f_{i_2} = a_{i_1} + \sum_{j \ne i_2} m_{i_2j} a_j = a_{i_2} + \sum_{j \ne i_1} m_{i_1j} a_j = \deg x_{i_2} f_{i_1}$$
because $\deg f_i = \sum_{j \ne i} m_{ij} a_j$ for all $1 \le i \le n$. 
\end{proof}

\section{Eagon-Northcott complexes}\label{ENcpx}

In this section, let us recall the Eagon-Northcott complex associated to a $2 \times n$ matrix, which plays an important role in the proof of Theorem \ref{MainThm}. Let $a_1, a_2, \ldots , a_n$ be positive integers and let $S=k[x_1,x_2, \ldots , x_n]$ be the polynomial ring over a field $k$. We regard $S$ as a $\Bbb Z$-graded ring so that $S_0=k$ and $\deg x_i = a_i$ for all $1\le i \le n$.  Let $f_1, f_2, \ldots,f_n \in S$ be homogeneous elements with $\deg f_i = b_i >0$. Suppose that $b_i - a_i$ is constant and independent of the choice of $i$. We set $\alpha = b_i - a_i$ and consider the matrix 
$$
A=\begin{pmatrix}
f_1 & f_2 & \cdots & f_n\\
x_1 & x_2 & \cdots & x_n
\end{pmatrix}.
$$
Let $F$ be a free $S$-module with rank $n$ and $\{T_i\}_{1 \le i \le n}$ a free basis. Let $K = \bigwedge F$ be the exterior algebra of $F$ over $S$. We denote by $\partial_1$ and $\partial_2$ the differentiations of the Koszul complexes $\rmK_\bullet(f_1, f_2, \ldots , f_n;S)$ and $\rmK_\bullet(x_1, x_2, \ldots , x_n;S)$, respectively. Let $U=S[y_1,y_2]$ be the polynomial ring over $S$, which we regard as a standard $\Bbb Z$-graded ring over $S$. To construct a chain complex, we set $C_0 = S$ and $C_q = K_{q+1} \otimes_S U_{q-1}$ for $1 \le q \le n-1$. Hence, for each $1 \le q \le n-1$, $C_q$ is a finitely generated free $S$-module with 
$$
\{T_\Lambda \otimes y_1^{q-1-\ell}y_2^\ell \mid \Lambda \subseteq \{1,2, \ldots , n\},\ \sharp \Lambda = q+1,\ 0 \le \ell \le q-1\}
$$
a free basis, where $T_\Lambda = T_{i_1} T_{i_2} \cdots T_{i_{q+1}}$ with $\Lambda = \{i_1 < i_2 < \cdots < i_{q+1}\}$. Then, the Eagon-Northcott complex associated to $A$ is defined to be the complex
$$
0 \to C_{n-1} \overset{d_{n-1}}{\to} C_{n-2} \overset{d_{n-2}}{\to} \cdots \to C_1 \overset{d_1}{\to} C_0 \to 0 \qquad (\mathrm{EN})
$$
of finitely generated free $S$-modules with the differentiations 
\begin{center}
$d_1(T_i T_j \otimes 1) = \det \begin{pmatrix} f_i & f_j\\x_i & x_j\end{pmatrix}$\ \ \ for \  \ $1 \le i < j \le n$ 
\end{center}and 
$$
d_q(T_\Lambda \otimes y_1^{q-1-\ell}y_2^\ell) =
\begin{cases}
\partial_1(T_\Lambda) \otimes y_1^{q-2-\ell}y_2^\ell + \partial_2(T_\Lambda) \otimes y_1^{q-1-\ell}y_2^{\ell-1} & \text{if }1 \le \ell \le q-2,\\
\partial_1(T_\Lambda) \otimes y_1^{q-2} & \text{if }\ell = 0,\\
\partial_2(T_\Lambda) \otimes y_2^{q-2} & \text{if } \ell = q-1.\\
\end{cases}
$$
With this notation, one can find in \cite{EN} that the above complex $(\mathrm{EN})$ gives rise to a free resolution of $R=S/\rmI_2(X)$, once $\height_S \rmI_2(X) = n-1$. To prove Theorem \ref{MainThm}, we use the graded structure of the complex $(\mathrm{EN})$ more closely. First of all, we regard each term of the complex $(\mathrm{EN})$ to be a $\Bbb Z$-graded $S$-module so that 
$$
\deg (T_\Lambda \otimes y_1^{q-1-\ell}y_2^\ell) = \sum_{i \in \Lambda} b_i -(\ell+1)\alpha,
$$
making $(\mathrm{EN})$ a complex of graded $S$-modules. Therefore, $C_{n-1} = \bigoplus_{i=1}^{n-1} S(i\alpha -b)$ as a graded $S$-module, where $b = \sum_{i=1}^n b_i$. Let $\rmK_S = S(-\sum_{i=1}^na_i)$ denote the graded canonical module of $S$. Then, if  $\height_S \rmI_2(X) = n-1$, by taking the $\rmK_S$-dual of the complex $(\mathrm{EN})$, we get a minimal presentation
$$
\bigoplus_{i=1}^{n-1}S(i\alpha) \to \rmK_{R} \to 0
$$
of the graded canonical module $\rmK_{R}$ of $R=S/\rmI_2(X)$, since  $\alpha = b_i - a_i$ for all $1 \le i \le n$.

\section{Proof of Theorem \ref{MainThm}}

The purpose of this section is to prove Theorem \ref{MainThm}. First of all, let us recall our notation. Let $n \ge 3$ be an integer and $H=\left<a_1, a_2, \ldots , a_n\right>$ a numerical semigroup. We assume that $H$ is minimally generated by the $n$ numbers $\{a_i\}_{1 \le i \le n}$.
Let $R=k[H]$ denote the semigroup ring of $H$ over a field $k$ and $S=k[x_1, x_2, \ldots, x_n]$ the polynomial ring. We regard $S$ as a $\Bbb Z$-graded ring so that $S_0 = k$ and $\deg x_i = a_i$ for all $1 \le i \le n$. Let 
$
\varphi : S \to R
$
be the homomorphism of graded $k$-algebras such that $\varphi(x_i) = t^{a_i}$ for all $1 \le i \le n$. We set $I=\Ker \varphi$.

We start the proof of Theorem \ref{MainThm}.

\begin{proof}[Proof of Theorem {\rm \ref{MainThm}}] \hspace{1cm} \\
 (2) $\Rightarrow$ (1) and (a). We have only to show assertion (a). Since $x_ix_i^{\ell_i}-x_{i-1}^{\ell_{i-1}}x_{i+1} \in I$ for $2 \le i \le n$ and $x_1x_1^{\ell_1}-x_2^{\ell_2}x_n \in I$, we readily get $(\ell_i + 1)a_i \in H_i$ for all $1 \le i \le n$, where $H_i = \left<a_1, \ldots , \overset{\vee}{a_i}, \ldots , a_n\right>$. Suppose that $\ell_i a_i \in H_i$ for some $1 \le i \le n$. Then $x_i^{\ell_i} - \zeta \in I$ with $\zeta \in (x_1, \ldots , \overset{\vee}{x_i}, \ldots , x_n)$. Therefore, by substituting 0 for all $\{x_j\}_{j \ne i}$, we get $x_i^{\ell_i} \in (x_i^{\ell_i+1})$, which is impossible. Thus $\ell_i = \min\{\ell >0 \mid \ell a_i \in H_i\} -1$ for all $1 \le i \le n$.

(1) $\Rightarrow$ (3) Note that $f_i \not\in I$ for all $1 \le i \le n$. In fact, suppose, say $f_1 \in I$. Then, since $x_1f_i \equiv x_1 f_i - x_if_1 \equiv 0 \ \mod~I$ and the ideal $I$ is prime, we get $f_i \in I$ for all $1 \le i \le n$, so that $I = (0)$, because $I \subseteq (x_i \mid 1 \le i \le n){\cdot}I$. This is absurd. Let $b_i=\deg f_i$ for $1\le i \le n$. Then, since $I$ is a graded ideal and since $x_if_j - f_ix_j \in I$  but $x_i f_j \notin I$ for any $1\le i < j \le n$, we see $a_i + b_j = a_j + b_i$, whence the number $b_i - a_i$ is independent of the choice of $1 \le i \le n$. We set $\alpha = b_i-a_i$ and now apply the observation in Section \ref{ENcpx}. Because $\height_S I=\dim S -\dim S/I=n-1$, the Eagon-Northcott complex associated to the matrix $
	\begin{pmatrix}
	f_1 & f_2 & \cdots & f_n\\
	x_1 & x_2 & \cdots & x_n
	\end{pmatrix} $ 
is a graded  minimal $S$-free resolution of $R$, which gives rise to a minimal graded presentation
$$
\bigoplus_{i=1}^{n-1}S(i\alpha) \to \rmK_R \to 0
$$
of the graded canonical module $\rmK_R$ of $R$, so that $$\PF(H) = \left\{\alpha, 2\alpha, \ldots , (n-1)\alpha\right\},$$ since $\rmK_R = \sum_{\alpha \in \PF(H)} Rt^{-\alpha}$. Therefore, by Proposition \ref{charAG}, $R=k[H]$ is an almost Gorenstein graded ring, which shows assertions (b), (c), and (d).

(3) $\Rightarrow$ (2) If $r \alpha \in H$ for some $1 \le r \le n-1$, then $(r+1)\alpha = \alpha + r \alpha \in H$ by definition of pseudo-Frobenius numbers, so that $(n-1) \alpha \in H$, which contradicts condition (3). Hence, $r \alpha \notin H$ for any $1 \le r \le n-1$. Because $\alpha + \beta \in \PF(H)$ if $\alpha, \beta\in \PF(H)$ and $\alpha + \beta \notin H$, we get $r \alpha \in \PF(H)$ for all $1 \le r \le n-1$, i.e.,  $\left\{\alpha, 2\alpha, \ldots, (n-1)\alpha \right\} \subseteq \PF(H)$. If $n=3$. then $\{\alpha, 2\alpha \} = \PF(H)$, so that by \cite{GMP, NNW} the implication (3) $\Rightarrow$ (2) follows (see Proposition \ref{charAG} also). Therefore, in order to show the implication, we may assume that $n \ge 4$ and that Theorem  \ref{MainThm} holds true for $n-1$.

Let $M = (m_{ij})$ be an RF-matrix associated to the pseudo-Frobenius number $\alpha$. We begin with the following.

\begin{proposition}\label{charofRF}
Every column of $M$ has at least one positive entry.
\end{proposition}

\begin{proof}
Assume the contrary, say $m_{i1} = 0$ for all $2 \le i \le n$. Let $d = \GCD(a_2, a_3, \ldots , a_n)$ and consider the following two equations
\begin{eqnarray*}
\alpha &=& (- a_1) + m_{12}a_2 + m_{13}a_3+\cdots + m_{1n}a_n \\
&=&m_{21}a_1+(-a_2) + m_{23}a_3 + \cdots + m_{2n}a_n
\end{eqnarray*}
arising from the first and the second rows of $M$. Then by the latter equation, $d ~|~\alpha$ since $m_{21}=0$, so that $d~|~a_1$ by the former one. Hence $d~|~a_i$ for all $1 \le i \le n$, so that $d=1$. We consider the numerical semigroup $H' = \left<a_2, a_3, \ldots , a_n\right>$; hence  $H'\subseteq H$. Then $\alpha \not\in H'$, but $\alpha + a_i \in H'$ for every $2 \le i \le n$, because 
$$\alpha=m_{i2}a_2 + \cdots +(-a_{i}) + \cdots + m_{in}a_n.$$ Hence $\alpha \in \PF(H')$. Therefore, because $(n-2)\alpha \notin H'$, the hypothesis of induction on $n$ shows $\PF(H') = \left\{\alpha, 2\alpha, \ldots ,(n-2)\alpha\right\}$, which is impossible because $$\rmf(H') = (n-2) \alpha < (n-1) \alpha \notin H'.$$ Thus $m_{i1} > 0$ for some $2 \le i \le n$.
\end{proof}

\begin{proposition}\label{formofRF}
	After suitable permutations of $a_1, a_2, \ldots ,a_n$ if necessary, we get
$$ M=\begin{pmatrix}
		 -1    & \ell_2 &     0  & \cdots  & 0 \\ 
		  0    & -1     & \ell_3 &   \cdots & 0 \\ 
		 \vdots &  & \ddots & \ddots & \vdots \\ 
		      0 &  \cdots    &      0 & -1 & \ell_n\\ 
		\ell_1 &    0   &  \cdots   &  0 & -1 
 \end{pmatrix}$$
for some integers $\ell_1, \ell_2, \ldots , \ell_n > 0$.
\end{proposition}

\begin{proof} 
Remember that for all $1 \le i \le n$, $m_{ii}=-1$  and 
$$
\alpha = \sum_{j=1}^n m_{ij}a_j \qquad (RF_i).
$$
Then we have the following.

\begin{claim}\label{claim3}
Let $1 \le \ell \le n-1$ be an integer. Then after suitable permutations of $a_1, a_2, \ldots, a_n$ if necessary, the following three assertions hold true for all $1 \le k \le \ell$.
\begin{enumerate}[{\rm (1)}]
\item $m_{k, k+1} >0$. 
\item $m_{i,k+1}=0$ for all $1 \le i \le  n$ such that $i \ne k,k+1$. 
\item $m_{i1}=0$ for all $2 \le i \le k+1$ such that $i \ne n$.
\end{enumerate}\end{claim}

\begin{proof}
We may assume that $1 \le \ell \le n-1$ and assertions (1), (2), and (3) hold true for $\ell-1$. Consider the sum of the equations $\{(RF_i)\}_{i \ne \ell}$. We then have  
$$
(n-1)\alpha = \sum_{j=1}^n \left(\sum_{i \ne \ell} m_{ij}\right) a_j \notin H,
$$
whence $\sum_{i \ne \ell} m_{ij} < 0$ for some $1 \le j \le n$. Suppose $j=1$. Then $\ell \ge 2$ and $m_{i1} = 0$ for all $i \ne 1,\ell$, so that $m_{\ell,1} > 0$ by Proposition \ref{charofRF}, which contradicts assertion (3). Therefore $j \ge 2$. Suppose $2 \le j \le \ell$. Then $j < \ell$ by Proposition \ref{charofRF}, because $\sum_{i \ne \ell} m_{ij} < 0$. Hence $m_{ij}=0$ for all $i \ne j, \ell$ and $m_{\ell, j} >0$. Consequently, assertions (1) and (2) imply that $\ell=j-1$, which is absurd. Hence $\ell +1 \le  j$ and therefore, $m_{ij}=0$ for $i \ne \ell, j$, whence $m_{\ell,j} >0$. Therefore, after exchanging $a_{\ell+1}$ for $a_j$ if necessary, assertions (1) and (2) follow for $k = \ell $.

To complete the proof of Claim \ref{claim3}, we must show assertion (3).  Because  $m_{i1}=0$ if $2 \le i \le \ell$, we have only to show that $m_{\ell+1,1} = 0$, provided $\ell \le n - 2$. Consider the sum of equations $\{(RF_i)\}_{1 \le i \le \ell +1}$, and we get 
$$
(\ell+1) \alpha = \left((-1) + m_{\ell+1,1}\right)a_1 + \sum_{i=2}^{\ell+1}(m_{i-1,i}-1)a_i + \sum_{j=\ell+1}^n(\sum_{i=1}^{\ell+1} m_{ij})a_j \not\in H, 
$$
whence $m_{\ell+1,1} = 0$, because $m_{i-1, i}> 0$ for all $2 \le i \le \ell+1$.\end{proof}

By Claim \ref{claim3}, after suitable permutations of $a_1, a_2, \ldots, a_n$ if necessary, the RF-matrix $M$ associated to $\alpha$  has the form 
$$ M=\begin{pmatrix}
		 -1    & \ell_2 &     0  & \cdots  & 0 \\ 
		  0    & -1     & \ell_3 &   \cdots & 0 \\ 
		 \vdots &  & \ddots & \ddots & \vdots \\ 
		      0 &  \cdots    &      0 & -1 & \ell_n\\ 
		m_{n1} &    0   &  \cdots   &  0 & -1 
 \end{pmatrix}$$
for some integers $\ell_2, \ell_3, \ldots , \ell_n > 0$. We then have $m_{n1} > 0$ by Proposition \ref{charofRF}, which completes the proof of Proposition \ref{formofRF}, setting $\ell_1 = m_{n1}$.
\end{proof}

Let us show that  $$ I=\rmI_2 \begin{pmatrix} x_2^{\ell_2}& \cdots &x_n^{\ell_n}&x_1^{\ell_1} \\ 
x_1& \cdots &x_{n - 1}&x_n 
\end{pmatrix}.$$
We set $ J=\rmI_2 \begin{pmatrix} x_2^{\ell_2}& \cdots &x_n^{\ell_n}&x_1^{\ell_1} \\ 
x_1& \cdots &x_{n - 1}&x_n 
\end{pmatrix}$. Then $J\subseteq I$ by Proposition \ref{relRF}. Note that $J$ is a perfect ideal in $S$ of height $n-1$, since $\sqrt{J+(x_1)}=(x_i \mid 1 \le i \le n)$. Therefore, the Eagon-Northcott complex 
$$
0 \to C_{n-1} \to C_{n-2} \to \cdots \to C_1 \to C_0 \to 0
$$
associated to the matrix
$\begin{pmatrix}
x_2^{\ell_2}& \cdots &x_n^{\ell_n}&x_1^{\ell_1}\\
x_1& \cdots &x_{n-1}&x_n
\end{pmatrix}$
gives rise to a graded minimal $S$-free resolution of $S/J$. Let 
$$
0  \to F_{n-1} \to \cdots \to  F_1 \to F_0 \to R \to 0
$$ be a graded minimal $S$-free resolution of $R=S/I$.  Let $\tau : S/J \to R$ be the canonical epimorphism and lift $\tau$ to a homomorphism 
$$
\xymatrix{
	0  \ar[r] & F_{n-1}\ar[r] &\cdots \ar[r] & F_1 \ar[r]& F_0 \ar[r]& R \ar[r] &0 \\
	0 \ar[r]&C_{n-1}\ar[r] \ar[u]^{f_{n-1}}&\cdots \ar[r] & C_1 \ar[r] \ar[u]^{f_1}& C_0 \ar[r] \ar[u]^{f_0}& S/J \ar[r]\ar[u]^\tau &0
}
$$
of complexes of graded $S$-modules. Let $\rmK_S=S(-\sum_{i=1}^na_i)$ denote the graded canonical module of $S$ and take the $\rmK_S$-dual of the above commutative diagram. Then we get a commutative diagram
$$
\xymatrix{
	&{F^\vee_{n-1}} {\ar[r]_\rho}{\ar[d]_\psi} &\rmK_R \ar[r] \ar[d]_{\varepsilon} &0 \\
	&{C^\vee_{n-1}} {\ar[r]_\eta} &\rmK_{S/J} \ar[r]& 0
}
$$
of minimal graded presentations of $\rmK_R$ and $\rmK_{S/J}$, where $(-)^\vee = \Hom_S(-, \rmK_S)$ and $\psi=\Hom_S(f_{n-1},\rmK_S)$ denotes the homomorphism induced from $f_{n-1}$. Remember that $i\alpha \in \PF(H)$ for all $1 \le i \le n-1$ and we identify $$C^\vee_{n-1}=\bigoplus_{i=1}^{n-1}S(i\alpha)\ \ \ \text{and}\ \ \ F^\vee_{n-1} =C_{n-1}^\vee \oplus \bigoplus_{\beta \in \PF(H) \setminus \{i\alpha \mid 1 \le i \le n-1\}}S(\beta)$$ (see Section \ref{ENcpx} and Proposition \ref{charAG}, respectively). Then the presentation matrix $\Bbb M$ of the homomorphism $\psi : F^\vee_{n-1} \to C^\vee_{n-1}$ of graded $S$-modules has the form
$$\Bbb M = 
\left(
\begin{array}{cccc|c}
c_1                    &      &           &           & \\
                        & c_2 &           &  \mbox{\Huge $*$}         &  \\
                        &      & \ddots &           & \mbox{\Huge $*$} \\
\mbox{\Huge $0$}&      &           & c_{n-1} &
\end{array}\right),$$
where $c_1, c_2, \ldots , c_{n-1} \in S_0=k$. On the other hand, the canonical exact sequence 
$$0\rightarrow I/J\rightarrow S/J\rightarrow R \rightarrow 0$$
induces a commutative diagram 
$$
\xymatrix{
0 \ar[r] & \Ext^{n-1}_S(R,K_S) \ar[r] \ar[d]^\cong & \Ext^{n-1}_S(S/J,K_S) \ar[r] \ar[d]^\cong & \Ext^{n-1}_S(I/J,K_S) \ar[r] &0\\
 & \rmK_{R} \ar[r]^\varepsilon & \rmK_{S/J} & &
}~(\sharp)
$$
with exact first row, because $I/J$ is a $1$-dimensional Cohen-Macaulay $S$-module, provided  $I/J \ne (0)$. Consequently, the homomorphism $\varepsilon : \rmK_R \to \rmK_{S/J}$ is injective. We furthermore have the following.

\begin{proposition}\label{claim4}
$c_ i \ne 0$ for any $1 \le i \le n-1$.
\end{proposition}

\begin{proof}
Suppose that $c_i=0$ for some $1 \le i \le n-1$ and choose such $i$ as small as possible. We may assume, after suitable elementary column transforms on $\Bbb M$, that $c_j = 1$ for all $1 \le j < i$ and that the first $i-1$ rows of $\Bbb M$ are all $\mathbf{0}$. Therefore, the $i$-th row of the matrix $\Bbb M$ has to be $\mathbf{0}$, i.e., $\psi(\mathbf{e}_i)=0$, where $\{\mathbf{e}_j\}_{1 \le j \le n-1}$ denotes the standard basis of $C^\vee_{n-1}=\bigoplus_{j=1}^{n-1}S(j\alpha)$. Consequently, because $\varepsilon(\rho(\mathbf{e}_i))=\eta(\psi(\mathbf{e}_i))=0$ and the homomorphism $\varepsilon$ is injective, we get $\rho(\mathbf{e}_i)=0$. This is, however, impossible, because $\rho(\mathbf{e}_i)$ is a part of a minimal system of generators of $\rmK_R$. Thus $c_i \ne 0$ for any $1 \le i \le n-1$.
\end{proof}

Consequently, after suitable elementary column transforms on $\Bbb M$, we may assume that $\Bbb M$ has the form
$$\Bbb M = 
\left(
\begin{array}{cccc|c}
1                    &      &           &           & \\
                        & 1&           &  \mbox{\Huge $0$}         &  \\
                        &      & \ddots &           & \mbox{\Huge $0$} \\
\mbox{\Huge $0$}&      &           & 1 &
\end{array}\right),$$ 
which guarantees that $\PF(H) = \{i\alpha \mid 1 \le i \le n-1\}$, because the homomorphism $$F^\vee_{n-1} =\bigoplus_{i=1}^{n-1}S(i\alpha)\oplus \bigoplus_{\beta \in \PF(H) \setminus \{i\alpha \mid 1 \le i \le n-1\}}S(\beta) \to \rmK_R \to 0$$ is minimal and the homomorphism $\varepsilon$ is injective. Therefore, $\psi: F_{n-1}^\vee \to C_{n-1}^\vee$ is an isomorphism, whence so is the homomorphism $\varepsilon : \rmK_R \to \rmK_{S/J}$. Thus $I=J$ because $\Ext_S^{n-1}(I/J,\rmK_S)= (0)$ by sequence $(\sharp)$, which completes the proof of Theorem \ref{MainThm}.
\end{proof}

\section{Examples}

To close this article, let us note a few examples. First, we give  Example \ref{ex1} (resp. Example \ref{ex2}) which satisfies (resp. does not satisfy) the conditions in Theorem \ref{MainThm}. Let $k$ be a field.

\begin{ex}\label{ex1}
Let $n$ and $\alpha$ be positive integers such that $n \ge 3$ and $\operatorname{GCD}(n, \alpha)=1$. We consider the numerical semigroup $H=\left<n, n+\alpha, n+2\alpha, \ldots , n+(n-1)\alpha \right>$. Then $\PF(H) = \{\alpha, 2\alpha, \ldots , (n-1)\alpha \}$ and $H$ satisfies condition (3) of Theorem \ref{MainThm}.
\end{ex}

\begin{ex}\label{ex2}
Let $H=\left<4,6,7,9\right>$. Then $\PF(H) = \{2,3,5\}$. The ring $R=k[H]$ is an almost Gorenstein graded ring (see Proposition \ref{charAG}), but does not satisfy the conditions in Theorem \ref{MainThm}. The defining ideal $I$ of $R$ is given by 
$$I=
\rmI_2 \begin{pmatrix}
x_2 & x_1^2 & x_4 & x_1x_3\\
x_1 & x_2 & x_3 & x_4
\end{pmatrix}+
\rmI_2 \begin{pmatrix}
x_3 & x_4 & x_1x_2 & x_1^3\\
x_1 & x_2 & x_3 & x_4
\end{pmatrix},
$$
which cannot be generated by the $2 \times 2$ minors of any $2 \times 4$ matrix.
We however have $\mu_S(I) = 6$, where $\mu_S(I)$ denotes the number of elements in a minimal homogeneous system of generators of $I$.
\end{ex}

If $n=4$ and the numerical semigroup $H$ satisfies the conditions in Theorem \ref{MainThm}, we then have $\mu_S(I) = 6$. However, even if $n=4$ and the semigroup ring $R=k[H]$ is an almost Gorenstein graded ring, the equality $\mu_S(I) = 6$ does not hold true in general. We note one example.

\begin{ex}
Let $H=\left<10,11,13,14\right>$. Then $\PF(H) = \{12,17,29\}$ and $R=k[H]$ is an almost Gorenstein graded ring. The defining ideal $I$ of $R$ is given by
$$I=
\rmI_2 \begin{pmatrix}
x_2^2 & x_1x_3 & x_2x_4 & x_3^2\\
x_1 & x_2 & x_3 & x_4
\end{pmatrix}+
\rmI_2 \begin{pmatrix}
x_3x_4 & x_4^2 & x_1^3 & x_1^2x_2\\
x_1 & x_2 & x_3 & x_4
\end{pmatrix} + (x_1x_4 - x_2x_3)
$$
and $\mu_S(I) = 7$.
\end{ex}



\end{document}